\documentclass[12pt]{amsart}

\usepackage{amsfonts}
\usepackage{amssymb}
\usepackage[letterpaper, left=2.5cm, right=2.5cm, top=2.5cm,
bottom=2.5cm,dvips]{geometry}
\usepackage{verbatim}
\usepackage[T1]{fontenc}
\usepackage{graphicx}
\usepackage{amsmath}

\usepackage{algorithm}
\usepackage[noend]{algpseudocode}

\newtheorem{theorem}{Theorem}
\theoremstyle{plain}

\newtheorem{claim}{Claim}

\newtheorem{conjecture}{Conjecture}
\newtheorem{corollary}{Corollary}

\newtheorem{lemma}{Lemma}

\newtheorem{problem}{Problem}
\newtheorem{proposition}{Proposition}

\numberwithin{equation}{section}
\usepackage{diagbox} 

\usepackage{array,multirow}

\usepackage{xcolor}

\usepackage[backref]{hyperref}
\hypersetup{
	colorlinks,
	linkcolor={blue!60!black},
	citecolor={red!60!black},
	urlcolor={red!60!black}
}

\usepackage{tikz}

\newcommand{\tikzcircle}[2][red,fill=red]{\tikz[baseline=-0.75ex]\draw[#1,radius=#2] (0,0) circle ;}

\begin{document}
	\title{Words avoiding tangrams}
	
		\author[M. D\k ebski]{Micha{\l} D\k ebski}
	\address{Faculty of Mathematics and Information Science, Warsaw University
		of Technology, 00-662 Warsaw, Poland}
	\email{michal.debski87@gmail.com}
	
	\author[J. Grytczuk]{Jaros\l aw Grytczuk}
	\address{Faculty of Mathematics and Information Science, Warsaw University
		of Technology, 00-662 Warsaw, Poland}
	\email{jaroslaw.grytczuk@pw.edu.pl}
	
	\author[B. Pawlik]{Bart\l omiej Pawlik}
	\address{Institute of Mathematics, Silesian University of Technology, 44-100 Gliwice, Poland}
	\email{bpawlik@polsl.pl}
	
	\author[J. Przyby\l o]{Jakub Przyby\l o}
	\address{AGH University of Krakow, al. A. Mickiewicza 30, 30-059 Krakow, Poland}
	\email{jakubprz@agh.edu.pl}
	
	\author[M. \'{S}leszy\'{n}ska-Nowak]{Ma\l gorzata \'{S}leszy\'{n}ska-Nowak}
	\address{Faculty of Mathematics and Information Science, Warsaw University
		of Technology, 00-662 Warsaw, Poland}
	\email{malgorzata.nowak@pw.edu.pl}

\begin{abstract}
	A \emph{tangram} is a word in which every letter occurs an even number of times. Such word can be cut into parts that can be arranged into two identical words. The minimum number of cuts needed is called the \emph{cut number} of a tangram. For example, the word $\mathtt{\color{red}{0102}\color{blue}{0102}}$ is a tangram with cut number one, while the word $\mathtt{\color{red}{01}\color{blue}{01023}\color{red}{023}}$ is a tangram with cut number two. Clearly, tangrams with cut number one coincide with the well known family of words, known as \emph{squares}, having the form $UU$ for some nonempty word $U$.
	
	A word $W$ \emph{avoids} a word $T$ if it is not possible to write $W=ATB$, for any words $A$ and $B$ (possibly empty). The famous 1906 theorem of Thue asserts that there exist arbitrarily long words avoiding squares over alphabet with just \emph{three} letters. Given a fixed number $k\geqslant 1$, how many letters are needed to avoid tangrams with the cut number at most $k$? Let $t(k)$ denote the minimum size of an alphabet needed for that purpose. By Thue's result we have $t(1)=3$, which easily implies $t(2)=3$. Curiously, these are currently the only known exact values of this function.
	
	In our main result we prove that $t(k)=\Theta(\log_2k)$. The proof uses \emph{entropy compression} argument and \emph{Zimin words}. By using a different method we prove that $t(k)\leqslant k+1$ for all $k\geqslant 4$, which gives more exact estimates for small values of $k$. The proof makes use of \emph{Dejean words} and a curious property of \emph{Gauss words}, which is perhaps of independent interest.

	\end{abstract}
	
	\maketitle
	
\section{Introduction}\label{Section Introduction}
The main aim of this paper is to extend the famous theorem of Thue \cite{Thue} on words avoiding squares. A \emph{square} is a word of the form $UU$, where $U$ is a nonempty word. A \emph{factor} of a word $W$ is a word $F$ occurring in $W$ as a contiguous block of letters, what can be written as $W=AFB$ for some (possibly empty) words $A$ and $B$. A word $W$ is \emph{square-free} if it does not contain any square factors.

It is easy to see that every binary word of length more than three must contain a square factor. However, a beautiful theorem of Thue \cite{Thue} (see \cite{BerstelThue,Lothaire}) asserts that over alphabet with only \emph{three} letters, one may construct arbitrarily long square-free words. This result inspired a huge amount of research, giving birth to an entire discipline known as \emph{combinatorics on words} (see \cite{BeanEM,BerstelPerrin,Lothaire,LothaireAlgebraic}).

Figuratively speaking, a square is a word that can be split into two identical words with just one \emph{cut}. For instance, the word $\mathtt{hotshots}$ is an example of a square with the obvious cutting in the middle: $\mathtt{\color{red}{hots}\color{black}|\color{blue}{hots}}$. A word $T$ is called a \emph{tangram} if each letter occurs an even number of times (possibly zero) in $T$. Clearly, any tangram can be split into factors (possibly single letters) that can be arranged into two identical words. What is the minimum number of cuts needed for that purpose? For instance, the word $\mathtt{tuteurer}$ demands four cuts. Indeed, the cutting $\mathtt{\color{red}t\color{black}|\color{blue}ute\color{black}|\color{red}u\color{black}|\color{blue}r\color{black}|\color{red}er}$ allows one to make two copies of the word $\mathtt{uter}$, and no less cuts will do the job (as can be checked by hand). We denote this fact by $\mu(\mathtt{tuteurer})=4$ and say that this word has \emph{cut number} four.

More formally, the \emph{cut number} $\mu(T)$ of a tangram $T$ is the least number $k\geqslant 1$ such that $T=F_1F_2\cdots F_{k+1}$, where $F_i$ are nonempty words satisfying $F_{\sigma(1)}\cdots F_{\sigma(j)}=F_{\sigma(j+1)}\cdots F_{\sigma(k+1)}$, for some permutation $\sigma$ of the set $\{1,2,\ldots, k+1\}$ and some $1\leqslant j\leqslant k$. It is also convenient to extend the definition of the cut number to arbitrary finite words by assuming that $\mu(W)=\infty$ whenever $W$ is not a tangram.

Inspiration to study this concept comes from the famous \emph{necklace splitting theorem} (see \cite{Matousek}), which may be stated as follows. An \emph{anagram} of a word $W$ is a word $U$ obtained by rearranging the letters in $W$. For instance, the words $\mathtt{triangle}$ and $\mathtt{integral}$ form a pair of mutual anagrams. Clearly, any tangram $T$ can be cut into pieces that can be arranged into a pair of mutual anagrams. What is the least number of cuts needed for that purpose? The necklace splitting theorem asserts that $q$ cuts are always sufficient, where $q$ is the number of distinct letters in $T$. It was first proved by Goldberg and West in \cite{Goldberg-West}. An elegant proof based on the celebrated Borsuk-Ulam theorem was found by Alon and West in \cite{Alon-West}. Other proofs, extensions, or variations can be found in \cite{Alon,AlonGML,GrytczukLubawski,JoicPZ,Lason,DeLonguevilleZivaljevic,VrecicaZivaljevic}.

Using the notion of cut number one may extend the concept of square-free words as follows. A word $W$ is \emph{$k$-tangram-free} if it does not contain factors with cut number at most $k$. In other words, if $W=AFB$, then $\mu(F)\geqslant k+1$. Let $t(k)$ denote the least size of an alphabet for which there are arbitrarily long $k$-tangram free words. By Thue's theorem \cite{Thue} we know that $t(1)=3$. It is also not hard to verify that $t(2)=3$ (we will provide a simple explanation in Section 3). Our main theorem gives an upper bound on this function.

\begin{theorem}\label{Theorem Jajko}
	For every $k\geqslant 3$, we have $t(k) \leqslant 1024\lceil\log_2 k + \log_2\log_2 k\rceil$.
\end{theorem}

The proof of this result is based on the entropy compression method (see, for example, \cite{BosekGNZ,GrytczukKM}) and the well known sequence of \emph{Zimin words}---a fundamental structure used in the pattern avoidance theory (see \cite{BeanEM,LothaireAlgebraic}). The sequence of Zimin words is defined recursively as follows: $Z_1=a$, $Z_2=aba$, $Z_3=abacaba$, and, in general, $Z_n=Z_{n-1}xZ_{n-1}$, where $x$ is a new letter. We are also using these words in a simple proof of the lower bound, $t(k)\geqslant \log_2k$ (see Section \ref{Section Lower}), which shows that $t(k)=\Theta(\log_2k)$. The multiplicative constant $1024$ in the theorem is chosen just for convenience in computations. Most probably it can be improved by more exact calculations or applications of some other related techniques, like the Lov\'{a}sz Local Lemma (see \cite{AlonSpencer}) or the Rosenfeld counting (see \cite{Rosenfeld1,Rosenfeld2}).

By using different approach we get a weaker upper bound, which, however, gives better estimate for $t(k)$ when $k$ is small (up to $k\leqslant 18426$).

\begin{theorem}\label{Theorem Main t(k)}
For every $k\geqslant 4$, we have $t(k)\leqslant k+1$.
\end{theorem}

The proof is based on a key observation concerning the cut number of \emph{Dejean words}, which are defined as follows. Let $\mathbb{A}$ be a fixed alphabet of size $r\geqslant 2$. For a word $W=w_1w_2\cdots w_n$, $w_i\in\mathbb{A}$, denote by $|W|=n$ the \emph{length} of $W$. Suppose that $F$ is a factor of $W$. If $F=w_iw_{i+1}\cdots w_{i+m-1}$, with $|F|=m$, then we say that $F$ \emph{occurs} in $W$ \emph{at position} $i$. The \emph{distance} between two different occurrences of $F$ in $W$, namely $F=w_{i}w_{i+1}\cdots w_{i+m-1}$ and $F=w_{j}w_{j+1}\cdots w_{j+m-1}$, with $i<j$, is defined as $j-i$. In particular, if $W$ is a square-free word, then the distance between any two different occurrences of a factor $F$ in $W$ is at least $|F|+1$.

In 1972 Fran\c{c}oise Dejean \cite{Dejean1972} started to investigate words with maximum possible distance between consecutive occurrences of same factors. She proved \cite{Dejean1972} that there exist arbitrarily long \emph{ternary} words in which every factor $F$ repeats at distance at least $\frac{4}{3}|F|$, which is in this respect the best possible strengthening of the theorem of Thue. She also found \cite{Dejean1972} that in the analogous problem for \emph{quaternary} words, the best possible bound for the respective distance of repeated factors is $\frac{5}{2}|F|$. This was confirmed by a sophisticated construction found by Pansiot in \cite{Pan1984}. For alphabets of size $r\geqslant 5$, Dejean made a conjecture \cite{Dejean1972} that the repeated factors distance in extremal words will be at least $(r-1)|F|$, which again is best possible. This conjecture stimulated a tremendous amount of research (see \cite{Moh2007, Oll1992, Pan1984}) culminating in a breakthrough result by Carpi \cite{Carpi2007}, who confirmed it for all alphabet sizes $r\geqslant 33$. The missing cases of
 $r$ were solved subsequently by Currie and Rampersad \cite{CurrieRampersad}, and independently by Rao \cite{Rao2011}. So, now we know that the following theorem is true.

\begin{theorem}[Dejean's Conjecture]\label{Theorem Dejean}
	For every $r\geqslant 5$, there exist arbitrarily long words $D$ over an alphabet of size $r$ such that every factor $F$ of $D$ repeats at distance at least $(r-1)|F|$.
\end{theorem}

Any word satisfying the above distance property will be called a \emph{Dejean word}. In order to prove Theorem \ref{Theorem Main t(k)}, we just look at the cut number of Dejean words. We will prove the following result, which, in view of Theorem \ref{Theorem Dejean}, immediately implies Theorem \ref{Theorem Main t(k)}.

\begin{theorem}\label{Theorem Dejean Cut Number}
	The cut number of any Dejean word $D$ over alphabet of size $r\geqslant 5$ satisfies $\mu(D)\geqslant r$.
\end{theorem}

The proof of this theorem is based on an intriguing property of \emph{Gauss words}, which are simply tangrams with every letter occurring exactly twice. For instance, $AABCBDECED$ is a Gauss word with five distinct letters. This specific type of words arose in Gauss' investigations \cite{Gauss} of closed self-crossing curves in the plane with crossing points of multiplicity exactly two. Traversing such a curve produces just such a word, called a \emph{Gauss code} of the curve (see Figure \ref{Figure GaussCurve}). Gauss tried to characterize words that may arise as codes of such curves. His plan was completed much later by Lov\'{a}sz and Marx \cite{LovaszMarx}, Rosenstiehl \cite{Rosenstiehl}, and de Fraysseix and Ossona de Mendez \cite{FraysseixOssonaDeMendez}.

\begin{figure}[ht]
	
	\begin{center}
		
		\resizebox{7cm}{!}{
			
			\includegraphics{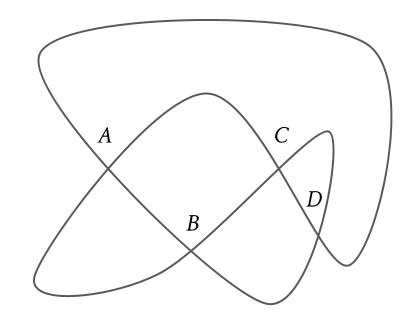}
			
		}
		\caption{An example of a closed self-crossing curve with four crossing points of multiplicity two whose Gauss code is $ACDABDCB$.}
		\label{Figure GaussCurve}
	\end{center}
\end{figure}

In the proof of Theorem \ref{Theorem Dejean Cut Number} we look at repeated factors in factorizations of words encoded by Gauss words. The details are explained in Section \ref{Section Gauss}. In Section \ref{Section Entropy} we provide a proof of Theorem \ref{Theorem Jajko}. The lower bound, together with some other simple facts, are proved in Section \ref{Section Lower}. The last section contains a brief discussion of possible future research.

\section{Gauss words and factorization patterns; proofs of Theorems \ref{Theorem Main t(k)} and \ref{Theorem Dejean Cut Number}}\label{Section Gauss}

Let $W$ be a word and let $W=F_1F_2\cdots F_n$ be a factorization of $W$ into non-empty factors. A word $P=p_1p_2\cdots p_n$ is called a \emph{pattern} of this factorization if, for each pair of indices $i\neq j$, we have $p_i=p_j$ if and only if $F_i=F_j$. For example, if $P=xyyzxz$, then $W=1234abcabcuv1234uv$ is a factorization with pattern $P$, as can be seen by applying substitutions $x=1234$, $y=abc$, and $z=uv$ to the word $P$, or by appropriate cutting: $$W=1234|abc|abc|uv|1234|uv.$$

Let us explain now the connection between Gauss words and the cut number on a simple example. Consider a tangram $T$ with the cut number $\mu(T)=4$. So, we have $$T=F_1F_2F_3F_4F_5,$$ and there is a permutation of the five factors giving two copies of the same word, say, $$F_3F_1=F_5F_2F_4.$$ Now, given the exemplary lengths of $F_i$'s, above equation leads to a more fragmented factorization of $T$, namely, $$T=DBABCCDA,$$ as shown in Figure \ref{Figure CuttingGauss}.

It should be clear that in general, any tangram $T$ with the cut number $k$ has a factorization whose pattern coincides with some Gauss word on at most $k$ distinct letters (depending on the arrangement of initial factors $F_i$ in the equation). We state this property more formally in the following lemma, whose easy proof is omitted.

\begin{figure}[ht]
	
	\begin{center}
		
		\resizebox{8cm}{!}{
			
			\includegraphics{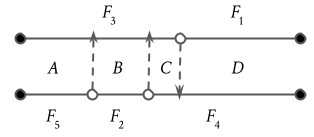}
			
		}
		\caption{Comparison of factors in equation $F_3F_1=F_5F_2F_4$ leading to identities $F_3=ABC$, $F_1=D$, $F_5=A$, $F_2=B$, $F_4=CD$, which give fragmented factorization $T=F_1F_2F_3F_4F_5=DBABCCDA$.}
		\label{Figure CuttingGauss}
	\end{center}
\end{figure}

\begin{lemma}\label{Lemma Cuts and Gauss Words}
	Every tangram $T$ has a factorization $T=G_1G_2\cdots G_{2s}$ whose pattern is a Gauss word $G=g_1g_2\cdots g_{2s}$, with $s\leqslant \mu(T)$.
\end{lemma} 

Factorizations of words whose patterns are Gauss words will be called shortly \emph{Gauss factorizations} with \emph{Gauss patterns}. Notice that one may produce plenty of words having Gauss factorizations simply by substituting arbitrary words for letters in arbitrarily chosen Gauss words as patterns.

Let $T$ be a tangram with Gauss factorization $T=G_1G_2\cdots G_{2s}$ having pattern $G=g_1g_2\cdots g_{2s}$. Let us call two equal factors $G_i=G_j$ a pair of \emph{twins}. So, $T$ consists of $s$ pairs of twins. We are going to prove that for at least one pair of twins, their relative distance cannot be too large. We will need two technical lemmas.

\begin{lemma}\label{Lemma Sum of Squares Inequality}
	Let $a_1, a_2,\ldots, a_s$ be any nonnegative real numbers. Then $$a_1^2+a_2^2+\cdots +a_s^2\geqslant \frac{1}{s}(a_1+a_2+\cdots+ a_s)^2,$$
	with equality if and only if $a_1=a_2=\cdots =a_s$.
\end{lemma}
This lemma is a direct consequence of the well-known Cauchy-Schwarz inequality, so, the proof is omitted.

Let $X=[a,b]$, $a\leqslant b$, be a non-empty segment of consecutive integers. We denote by $l(X)$ the \emph{geometric length} of $X$, that is, $l(X)=b-a=|X|-1$. We will also use a standard notation $[n]=\{1,2,\ldots,n\}$ for the initial segment of $n$ positive integers.

\begin{lemma}\label{Lemma Interwals on 2n Balls}
	Let $n\geqslant 1$ be an integer, and suppose that $A_1,A_2,\ldots,A_n\subseteq [2n]$ is a collection of segments of consecutive integers whose ends form a partition of the set $[2n]$. Then $$l(A_1)+l(A_2)+\cdots+l(A_n)\leqslant n^2.$$
\end{lemma}

\begin{proof} 
	The proof is by induction on $n$. It is straightforward to verify that the statement holds for some initial values of $n$. For instance, for $n=2$, it suffices to examine the following three cases:
	
	\begin{itemize}
		\item $A_1=\{1,2\}$, $A_2=\{3,4\}$.
		
		This partition can be depicted as $\tikzcircle[fill=white]{3pt}\,\tikzcircle[fill=white]{3pt}\,\tikzcircle[fill=black]{3pt}\,\tikzcircle[fill=black]{3pt}$. We have $l(A_1)+l(A_2)=1+1\leqslant2^2$.
		
		\item $A_1=\{1,2,3\}$, $A_2=\{2,3,4\}$.
		
		Here partition looks as $\tikzcircle[fill=white]{3pt}\,\tikzcircle[fill=black]{3pt}\,\tikzcircle[fill=white]{3pt}\,\tikzcircle[fill=black]{3pt}$, and we have $l(A_1)+l(A_2)=2+2\leqslant2^2$.
		
		\item $A_1=\{1,2,3,4\}$, $A_2=\{2,3\}$.
		
		Similarly, for $\tikzcircle[fill=white]{3pt}\,\tikzcircle[fill=black]{3pt}\,\tikzcircle[fill=black]{3pt}\,\tikzcircle[fill=white]{3pt}$, the statement holds as well: $l(A_1)+l(A_2)=3+1\leqslant2^2$.
	\end{itemize}
	
	Let us assume that the statement holds for $n-1$, that is, the sum of lengths of $n-1$ segments satisfying the assumption of the theorem is not greater than $(n-1)^2$. Consider $n$ segments, $A_1,\,A_2,\,\ldots,\,A_n$ in $[2n]$, with ends forming a partition of $[2n]$. We may assume that the right end of the last segment $A_n$ is $2n$, so, $A_n=[2n-k,2n]$ for some $k=1,2,\ldots,2n-1$:
	$$\underbrace{\tikzcircle[fill=white]{3pt}\,\tikzcircle[fill=white]{3pt}\,\cdots\,\tikzcircle[fill=white]{3pt}}_{2n-k-1}\,\tikzcircle[fill=black]{3pt}\,\underbrace{\tikzcircle[fill=white]{3pt}\,\tikzcircle[fill=white]{3pt}\,\cdots\,\tikzcircle[fill=white]{3pt}\,\tikzcircle[fill=black]{3pt}}_{k}.$$
	Notice that by the inductive assumption we may write $$l(A_1)+l(A_2)+\ldots+l(A_{n-1})\leqslant (n-1)^2+2n-k-1.$$ Indeed, after deleting the two black endpoints of $A_n$ and shifting the white circles into a segment, the length of each segment $A_i$ either remains unchanged or decreases by one. The later option may happen only if the two endpoints of $A_i$ were in two distinct groups of white circles (separated by the middle black one). Hence, the number of such segments is at most the size of the left group of white circles, namely $2n-k-1$. Since $l(A_n)=k$, we get
	$$l(A_1)+l(A_2)+\ldots+l(A_{n-1})+l(A_n)\leqslant (n-1)^2+2n-k-1+k=n^2,$$
	which completes the proof.
\end{proof}

We are now ready to state and prove the key property of Gauss factorizations of words.

\begin{theorem}\label{Theorem Twin-Distance Discrete}
	Let $T=G_1G_2\cdots G_{2s}$ be a Gauss factorization of a tangram $T$ with Gauss pattern $G=g_1g_2\cdots g_{2s}$. Then there exists a pair of twins $G_i=G_j=X$ within distance at most $s|X|$. Moreover, this distance is strictly smaller than $s|X|$, unless $G$ is a square and all factors $G_m$, with $1\leqslant m\leqslant 2s$, are of the same length.
\end{theorem}

\begin{proof} Let $T=t_1t_2\cdots t_{2n}$ be a tangram of length $2n$ and let $T=G_1G_2\cdots G_{2s}$ be its Gauss factorization with pattern $G=g_1g_2\cdots g_{2s}$. Suppose that $G_i=G_j=X$ is a pair of twins occurring at positions $p_i$ and $p_j$, respectively, with $p_i<p_j$. Consider then a collection $\mathcal{A}_X$ of segments in $[2n]$ of the form $[p_i+h,p_j+h]$, for $h=0,1,2,\ldots, |X|-1$. Clearly, each segment in the collection has length equal to the twin-distance of the pair $(G_i,G_j)$. So, the total sum of lengths of all segments in the collection equals $|X|(p_j-p_i)$.
	
	Suppose now that for each twin pair $G_i=G_j=X$ in $T$, the twin-distance satisfies  $p_j-p_i\geqslant s|X|$. Then we have $|X|(p_j-p_i)\geqslant s|X|^2$. Let us enumerate all twin pairs in $T$ as $$G_{i_1}=G_{j_1}=X_1,G_{i_2}=G_{j_2}=X_2,\ldots, G_{i_s}=G_{j_s}=X_s,$$and their corresponding segment families as $\mathcal{A}_{X_1},\mathcal{A}_{X_2},\ldots,A_{X_s}$. Clearly, all the endpoints of these segments are distinct and their total length is $$S=\sum_{h=1}^{s}|X_h|(p_{j_h}-p_{i_h})\geqslant s\sum_{h=1}^{s}|X_h|^2\geqslant s\cdot \frac{1}{s}\left(\sum_{h=1}^{s}|X_h|\right)^2=n^2,$$
	where the last inequality follows from Lemma \ref{Lemma Sum of Squares Inequality}. On the other hand, by Lemma \ref{Lemma Interwals on 2n Balls} we know that $S\leqslant n^2$. So, again by Lemma \ref{Lemma Sum of Squares Inequality}, we conclude that either there is a twin pair $G_i=G_j=X$ with distance $p_j-p_i<s|X|$, or all factors $X_h$ have the same length. In the later case, if $G$ is not a square, then there is a pair of letters $g_i=g_j$ with $j-i<s$. Hence, the corresponding pair of twins, $G_i=G_j=X$, is again at distance $p_j-p_i<s|X|$. This completes the proof.
\end{proof}

From the above result we easily derive proofs of Theorems \ref{Theorem Dejean Cut Number} and \ref{Theorem Main t(k)}.

\begin{proof}[Proof of Theorem \ref{Theorem Dejean Cut Number}]
	Let $D$ be any Dejean word over alphabet of size $r\geqslant 5$. Suppose that $\mu(D)\leqslant r-1$. Then $D$ must be a tangram and by Lemma \ref{Lemma Cuts and Gauss Words}, $D$ has a Gauss factorization $D=G_1G_2\cdots G_{2s}$, with $s\leqslant r-1$. We may assume that the pattern of this factorization is not a square, since otherwise $D$ would be a square itself, which cannot happen for Dejean words. So, by Theorem \ref{Theorem Twin-Distance Discrete}, there is a twin pair $G_i=G_j=X$ in $D$ with distance strictly less than $s|X|\leqslant (r-1)|X|$. This contradicts the defining property of Dejean words, so, it must be $\mu(D)\geqslant r$.
\end{proof}

\begin{proof}[Proof of Theorem \ref{Theorem Main t(k)}]
	Let $k\geqslant 4$ be a fixed integer. By Theorem \ref{Theorem Dejean}, there exists arbitrarily long Dejean words $D$ over alphabet of size $k+1$. Clearly, every factor of a Dejean word is also a Dejean word. Thus, by Theorem \ref{Theorem Dejean Cut Number}, every tangram factor $F$ of $D$ satisfies $\mu(F)\geqslant k+1$. So, $D$ is a $k$-tangram-free word, which proves that $t(k)\leqslant k+1$.
\end{proof}

	\section{Entropy compression and Zimin words; proof of Theorem \ref{Theorem Jajko}}\label{Section Entropy}
	
	The proof of Theorem \ref{Theorem Jajko} splits into two parts. In the first part we take care of sufficiently long tangrams with bounded cut number. In the first lemma we prove that they are avoidable over alphabet with $1024$ letters.
\begin{lemma}\label{lemma_avoid_long_tangrams}
	For every $k \geqslant 3$, there exist arbitrarily long words over alphabet of size $1024$ whose factors of length at least $k\log_2 k$ have cut number at least $k+1$.
\end{lemma}
\begin{proof}
	Fix $k\geqslant 3$ and an alphabet $\mathbb{A}$ of size $1024$. For a natural number $N$, let $f(N)$ denote the number of words of length at most $N$ over $\mathbb{A}$ that do not contain factors of length at least $k\log_2 k$ with cut number at most $k$. Note that, in order to prove the lemma, it suffices to show that $f(N)$ goes to infinity with $N$. From now on we assume that $N$ is fixed and head towards the conclusion that $f(N)\geqslant 2^N$.
	
	Consider Algorithm \ref{algorithm_compression} that encodes a sequence of length $N$ over $\mathbb{A}$.
	
	\begin{algorithm}
		\caption{Encoding of a sequence of length $n$ over some alphabet $\mathbb{A}$ of size $1024$.}
		\label{algorithm_compression}
		\begin{algorithmic}[1]
			\Procedure{Encode}{$x_1, x_2, \ldots, x_N$ -- a sequence over $\mathbb{A}$}
			\State $s \leftarrow $ empty word
			\State $L \leftarrow $ empty log
			\For{$i=1, 2, \ldots, N$}
			\State\label{step_add_letter} add $x_i$ at the end of $s$
			\If{$s$ contains a suffix $F$ of length least $k\log_2 k$ with cut number at most $k$}\label{step_erasurecondition}
			\State\label{step_information_saved} save $i$ and $F$ in $L$
			\State\label{step_remove_tangram} remove from $s$ the suffix $F$
			\EndIf
			\EndFor
			\State\Return $s$ and $L$
			\EndProcedure
		\end{algorithmic}
	\end{algorithm}
	
	We start by observing that the algorithm encodes uniquely the input sequence.
	
	\begin{claim}
		\label{claim_injective_function}
		Algorithm \ref{algorithm_compression} computes an injective function, i.e., for every two distinct inputs, the algorithm produces distinct outputs.
	\end{claim}
	
	To prove the claim note that the output is the state of the procedure after the last iteration of the main loop. Now, given the state of the procedure after the $j$-th iteration of the main loop, we can reconstruct $x_j$ and the state of the algorithm before the $j$-th iteration as follows. If $L$ does not contain an entry for step $j$, then $x_j$ is the last letter of $s$, and erasing it from the end of $s$ produces the state before the $j$-th iteration of the  main loop. Otherwise, $x_j$ is the last letter of the last suffix $F$ stored in $L$, and the state before the $j$-th iteration of the main loop is produced by appending to $s$ the word $F$ minus the last letter, and erasing the last entry from $L$. Therefore, by backwards induction on $j$, we conclude that the input sequence can be uniquely determined from its output, which completes the proof of Claim \ref{claim_injective_function}.
	
	Now we proceed to the crucial claim that gives an upper bound on the number of possible outputs returned by the algorithm.
	
	\begin{claim}
		\label{claim_number_of_outputs}
		The number of possible outputs of Algorithm \ref{algorithm_compression} is at most
		$$
		16^N 1024^{\frac{N}{2}}f(N).
		$$
	\end{claim}
	
	Note that each entry saved to $L$ in line \ref{step_information_saved} of Algorithm \ref{algorithm_compression} can be encoded by the following information.
	
	\begin{enumerate}
		\item The number $i$ of the current step.
		\item The length $\ell_i$ of the removed suffix $F$.
		\item Positions of $k$ cuts in $F$ that partition $F$ into $k+1$ nonempty words $F_1F_2, \ldots, F_{k+1}$.
		\item A permutation $\sigma$ of the set $[k+1]$ such that $F_{\sigma(1)}\cdots F_{\sigma(j)}=F_{\sigma(j+1)}\cdots F_{\sigma(k+1)}$ for some $j$.
		\item\label{loginfo_halves} A sequence of $\frac{\ell_i}{2}$ letters from $\mathbb{A}$ containing $F_{\sigma(1)}\cdots F_{\sigma(j)}$.
	\end{enumerate}

Notice that sometimes the number of cuts in the removed suffix $F$ can be smaller than $k$, but in our encoding we do not assume that all pieces $F_i$ must change their positions. This will not affect our analysis below, since the length of deleted tangrams is at least $k\log_2 k$.
	
	Now we will count the number of possible outputs of the algorithm, considering those types of information stored throughout the whole run, in the above order.
	
	\begin{enumerate}
		\item The set of steps $i$ that are saved in $L$ is a subset of $[N]$, so there are $2^N$ possibilities.
		\item The sequence of consecutive $\ell_i$'s is a sequence of positive integers that sum up to at most $N$ (because the total number of letters removed from $s$ in line \ref{step_remove_tangram} cannot exceed the total number of letters added in line \ref{step_add_letter} of the algorithm). Such a sequence can be thought of as a segment $[0,N]$ cut into fragments of length $\ell_1, \ell_2, \ldots$ (with the last segment of length $N$ minus the sum of $\ell_i$'s, if the sum is less than $N$). There are $N-1$ possible points at which the segment can be cut, so the total number of ways in which it can be cut is at most $2^{N-1}$; we upper bound it by $2^N$.
		\item For a single suffix $F$ of length $\ell_i$, the positions of cuts can be thought of as one of $2^{\ell_i - 1}$ sets of integer points in the open interval $(0, \ell_i)$. Note that, having fixed all $\ell_i$'s, we can store all this information as a single set of integer points in the open interval $(0, \sum_i \ell_i)$; and, as remarked earlier, $\sum_i \ell_i$ is at most $N$. Therefore, having fixed the lengths of all the suffixes stored in $L$, the number of possible sequences of cut positions is at most $2^N$.
		\item The number of permutations of $k+1$ elements is $(k+1)!$, which is at most $k^k$ for $k\geqslant 3$. Now, note that each entry of $L$ corresponds to removing at least $k \log_2 k$ letters from $s$, which implies that $L$ stores at most $\frac{N}{k \log_2 k}$ entries, so the number of possibilities is at most
		$$\left(k^k\right)^{\frac{N}{k \log_2 k}} = 2^N.$$
		\item Having fixed all $\ell_i$'s, all the information in (\ref{loginfo_halves}) can be represented as a single sequence over $\mathbb{A}$ of length at most $\frac{N}{2}$, which means that the number of possibilities is at most $1024^{\frac{N}{2}}$.
	\end{enumerate}
	
	By taking the product of all five above estimations we conclude that the number of possible logs $L$ returned by the algorithm is at most $16^N\cdot1024^{\frac{N}{2}}$. Moreover, the algorithm clearly returns a word $s$ that has length at most $N$ and does not contain a factor of length smaller than $k\log_2 k$ with cut number at most $k$, and by the definition of $f$ there are $f(N)$ such words. Therefore, the proof of Claim \ref{claim_number_of_outputs} is complete.
	
	Clearly the number of possible inputs to Algorithm \ref{algorithm_compression} is $1024^N$. Therefore, Claims \ref{claim_injective_function} and \ref{claim_number_of_outputs} imply that
	$$
	16^N\cdot1024^{\frac{N}{2}}\cdot f(N) \geqslant 1024^N.
	$$
	It follows that
	$$
	f(N) \geqslant \left( \frac{1024}{16\sqrt{1024}}\right)^N = 2^N,
	$$
	which completes the proof of the lemma.
\end{proof}

For the second part of the proof of Theorem \ref{Theorem Jajko}, we need to construct words avoiding short tangrams. As mentioned in the introduction, we will use Zimin words in the following setting. Let $\mathbb{A}=\{a_1,a_2,\ldots\}$ be a countably infinite alphabet of letters. The sequence of \emph{Zimin words} over $\mathbb{A}$ is defined recursively by taking $Z_1=a_1$ and $Z_{n}=Z_{n-1}a_nZ_{n-1}$, for every $n\geqslant 2$.

\begin{lemma}
	\label{lemma_avoid_short_tangrams}
	For every $q\geqslant1$ there exist arbitrarily long words over alphabet of size $q$ whose shortest tangrams have length at least $2^q$.
\end{lemma}
\begin{proof}
	For $q=1$ the assertion is trivially true. Let $q\geqslant2$ be fixed. Consider an infinite periodic word $W=Z_{q-1}a_qZ_{q-1}a_qZ_{q-1}a_q\cdots$ over alphabet $\mathbb{A}_q=\{a_1, a_2, \ldots, a_q\}$, where $Z_q$ is a Zimin word. We will prove that every tangram in $W$ has length at least $2^q$. 
	
	We start with stating a simple property of Zimin words, which follows immediately from the definition of $Z_n$.
	 
	\begin{claim}
		\label{claim_notangramsinzimin}
		For every $n\geqslant 1$, the Zimin word $Z_n$ has length $2^n-1$ and does not contain tangrams.
	\end{claim}
	
	Next, we will show that every tangram in $W$ containing all letters of $\mathbb{A}_q$ must be of length at least $2^q$. Below we formulate a stronger statement that we will prove inductively. For convenience, we denote $\mathbb{A}^{(i)}_q=\{a_i, a_{i+1},\ldots, a_q \}$, for $i=1,2,\ldots, q$.
	
	\begin{claim}
		\label{claim_longtangramsinW}
		Let $1\leqslant i\leqslant q-1$ be fixed. Suppose that $F$ is a factor of $W$ containing each letter of $\mathbb{A}^{(i)}_q$ a positive even number of times. Then the length of $F$ satisfies $$|F|\geqslant 3\cdot 2^{q-2}+1+\sum_{j=i}^{q-2}2^{j-1}.$$
	\end{claim}
	We will prove the claim by backward induction on $i$. Let $F$ be a factor of $W$. For the base case, $i=q-1$, let $F$ be a factor containing letters $a_q$ and $a_{q-1}$. Note that between each occurrence of $a_q$ and $a_{q-1}$ in $W$ there is a factor $Z_{q-2}$. Since $F$ must contain at least two letters $a_q$ and two letters $a_{q-1}$, it must contain also three disjoint copies of $Z_{q-2}$ between them. So, the length of $F$ must be at least $4 + 3\cdot\left(2^{q-2} - 1\right)= 3\cdot 2^{q-2} + 1$, as desired.
	
	Now suppose that the claim is true for $i+1$ and consider a factor $F$ of $W$ whose letters from $\mathbb{A}^{(i)}_q$ occurr a positive even number of times. Let $F'$ be the shortest factor of $F$ preserving this property. Note that $F'$ must start and end with a letter from $\mathbb{A}^{(i)}_q$, as otherwise it would not be the shortest. Moreover, between every two letters from $\mathbb{A}^{(i)}_q$ in $F'$ there is a copy of $Z_{i}$. Since $F'$ contains an even number of letters from $\mathbb{A}^{(i)}_q$, it follows that it must contain an odd number of $a_i$'s. Therefore, $F$ must contain $F'$ together with at least an additional copy of $a_i$ and a factor $Z_{i-1}$ that separates every two letters from $\mathbb{A}^{(i)}_q$. Hence, the length of $F$ is at least the length of $F'$ plus $2^{i-1}$.
	
	By the induction hypothesis, the length of $F'$ is at least $$3\cdot 2^{q-2} + 1 + \sum_{j=i+1}^{q-2}2^{j-1}.$$ So, the length of $F$ is at least $$3\cdot 2^{q-2} + 1 + \sum_{j=i+1}^{q-2}2^{j-1} + 2^{i-1} = 3\cdot 2^{q-2} + 1 + \sum_{j=i}^{q-2}2^{j-1},$$ as desired. The proof of the claim is therefore complete by induction.
	
	The above two claims imply that $W$ does not contain tangrams shorter than $2^q$. Indeed, by Claim \ref{claim_notangramsinzimin}, every tangram in $W$ must contain $a_q$, and since every two occurrences of $a_q$ in $W$ are separated by $Z_{q-1}$, such a tangram must contain each symbol from $\mathbb{A}_q$. Therefore, applying Claim \ref{claim_longtangramsinW} for $i=1$, we conclude that the length of this tangram must be at least 
	$$3\cdot 2^{q-2} + 1 + \sum_{j=1}^{q-2}2^{j-1} = 3\cdot 2^{q-2} + 1 + 2^{q-2} - 1 = 2^q,$$
	which completes the proof of the lemma.
\end{proof}

Using the above two lemmas we will get the assertion of Theorem \ref{Theorem Jajko} by a suitable product construction.

\begin{proof}[Proof of Theorem \ref{Theorem Jajko}]
	Fix $k\geqslant3$. We will construct a $k$-tangram-free word of any given length $n$ over an alphabet of size $1024\lceil\log_2 k + \log_2\log_2 k\rceil$.
	
	By Lemma \ref{lemma_avoid_long_tangrams}, there exists a word $V=v_1v_2\cdots v_n$ over some alphabet $\mathbb{A}$ of size $1024$ whose all factors $F$, with $|F|\geqslant k\log_2 k$, satisfy $\mu(F)\geqslant k+1$. By Lemma \ref{lemma_avoid_short_tangrams}, applied with $q=\lceil\log_2 k + \log_2\log_2 k\rceil$, there exist a word $W=w_1w_2\cdots w_n$ over an alphabet $\mathbb{B}$ of size $q$ whose all factors $F$, with $|F|<k\log_2 k\leqslant 2^q$, satisfy $\mu(F)=\infty$.
	
	We define the word $X=x_1x_2\ldots x_n$ over the alphabet $\mathbb{A} \times \mathbb{B}$ such that $x_i=(v_i, w_i)$. Note that if $X$ contains any factor $F=x_j x_{j+1} \ldots x_{j+\ell -1}$ with cut number at most $k$ and length $|F|=\ell$, then both words, $F_V=v_jv_{j+1}\ldots v_{j+\ell -1}$ and $F_W=w_jw_{j+1}\ldots w_{j+\ell -1}$ must be factors with cut number at most $k$ and length $\ell=|F_V|=|F_W|$ in $V$ and $W$, respectively. If $\ell\geqslant k\log_2 k$, it is impossible by the choice of $V$, and in the other case, when $\ell<k\log_2 k$, it is impossible by the choice of $W$. Therefore $X$ is the desired $k$-tangram-free word of length $n$ over an alphabet of size $1024q$, which completes the proof. 
\end{proof}

\section{Lower bound on $t(k)$ and other little things}\label{Section Lower}
In the proof of Lemma \ref{lemma_avoid_short_tangrams} we used a simple fact that Zimin words $Z_n$ have length $2^n-1$ and are tangram-free. It occurs that these are the longest words with this property over a fixed alphabet.

\begin{proposition}\label{Proposition Tangram q}
Let $q\geqslant 1$ be an integer. Every word of length $2^q$ over alphabet of size $q$ contains a tangram.
\end{proposition}
\begin{proof}
Let $W=w_1w_2\cdots w_n$ be a word over alphabet $\mathbb{A}=\{a_1,a_2,\ldots,a_q\}$. For every letter $a_i\in \mathbb{A}$, let $W(a_i)$ denote the number of times the letter $a_i$ occurs in $W$. Let $N_W=(W(a_1),W(a_2),\ldots, W(a_q))$ and let $N'_W$ be the corresponding vector in $\mathbb{Z}_2^q$ whose coordinates are the residues of numbers $W(a_i)$ modulo $2$. Notice that the word $W$ is a tangram if and only if $N'_W$ is the zero vector.

Suppose now that $n=2^q$ and consider all factors $F_h$ of the word $W$ defined by $F_h=w_1w_2\cdots w_h$, for $h=1,2,\ldots, 2^q$. If for some $h$, $N'_{F_h}$ is the zero vector, then we are done. Otherwise, there must be two distinct factors, $F_i$ and $F_j$, with $i<j$, such that $N'_{F_i}=N'_{F_j}$. But then, the factor $T=w_{i+1}\cdots w_j$ must be a tangram. Indeed, since $F_j=F_iT$, we have $N'_{F_j}=N'_{F_i}+N'_T$, and $N'_T$ must be the zero vector.
\end{proof}

To get the logarithmic lower bound for the function $t(k)$ from the above result it suffices to use a trivial fact that the cut number of any tangram $T$ satisfies $\mu(T)\leqslant |T|-1$.

\begin{corollary}\label{Corollary Lower Bound}
For every $k\geqslant 1$, we have $t(k)\geqslant \log_2(k+2)$.
\end{corollary}
\begin{proof}
Let $k\geqslant 1$ be fixed. Suppose, on the contrary, that $t(k)<\log_2(k+2)$. This means that there exist arbitrarily long words over alphabet of size $t(k)=q<\log_2(k+2)$ whose all factors $F$ satisfy $\mu(F)\geqslant k+1$. Let $W$ be such a word of length $2^q$. By Proposition \ref{Proposition Tangram q}, there is a tangram factor $F$ of $W$. Clearly, its cut number satisfies $$\mu(F)\leqslant |F|-1\leqslant |W|-1=2^q-1<k+1,$$
which proves the asserted inequality.
\end{proof}

As mentioned in Introduction, the only known exact values of the function $t(k)$ are for $k=1$ and $k=2$. In both cases we have $t(1)=3$ and $t(2)=3$, as a consequence of the theorem of Thue \cite{Thue}. The first equality is immediate. The second one follows from a simple fact that a square-free word cannot have cut number exactly two. Indeed, assume that $W$ is square-free and $\mu(W)=2$. This means that $W=ABC$ and at least one of the following six identities is satisfied: $$A=BC,A=CB,B=AC,B=CA,C=AB,C=BA.$$It is easy to check that substituting each of these identities to $W$ produces at least one square. Indeed, in the first two cases we get $W=ABC=BCBC$ and $W=ABC=CBBC$, in the next two we obtain $W=ABC=AACC$ and $W=ABC=ACAC$, and the last two give $W=ABC=ABAB$ and $W=ABC=ABBA$.

In case of $k=3$ we only know that $t(3)=4$ or $5$.

\begin{proposition}\label{Proposition t(3)}
	We have $4\leqslant t(3)\leqslant 5$.
\end{proposition}
\begin{proof}
	The upper bound follows from Theorem~\ref{Theorem Main t(k)}, since $t(3)\leqslant t(4)\leqslant 5$. 
	
	To verify the lower bound, suppose that $W$ is a sufficiently long square-free word over alphabet $\{a,b,c\}$. Then $W$ must contain a \emph{palindrome} of the form $aba$ in every factor of length six. Indeed, the only way to avoid palindromes using three letters is to construct a periodic word $abcabcabc\cdots$. Hence, the longest square-free and palindrome-free word over three letters has the form $abcab$. The palindromic factor $aba$ can be extended from both sides only by appending the letter $c$, giving a longer palindromic factor $cabac$ in $W$. Now further extension from both sides gives either the Zimin word $Z_3=acabaca$, or the word with prefix $bcabac$ or suffix $cabacb$. The word $Z_3$ cannot be further extended without creating a square. The other two words have cut number three, which is seen in cuttings $bca|b|a|c$ and $cab|a|c|b$. It follows that any sufficiently long ternary word must contain a tangram with cut number at most three. It follows that $t(3)\geqslant 4$.
\end{proof}

\section{Closing remarks}
Let us conclude the paper with posing some open problems and sketching a more general landscape for future investigations.

Theorem \ref{Theorem Jajko} and Corollary \ref{Corollary Lower Bound} imply together that $t(k)=\Theta(\log_2k)$. It is natural to wonder how close is actually $t(k)$ to the function $\log_2k$.

\begin{problem}
Is there a constant $C$ such that $t(k)\leqslant \log_2k+C$?
\end{problem}

It would be also nice to know more on some small values of $t(k)$. In particular, whether $t(3)=4$ or $5$.

\begin{problem}
	Determine $t(3)$.
\end{problem}

One tempting approach could be to look at the corresponding quaternary Dejean words. However, the example found by Pansiot in \cite{Pan1984} is not $3$-tangram-free. Indeed, the infinite word $N$ from \cite{Pan1984} having the desired extremal property starts with $$N=abcadbacdabcdacbdcadbacdabca\cdots.$$ Unfortunately, at position $9$ it contains a factor $dabcdacb$ whose cut number is $3$: $$dabc|da|c|b.$$

Referring to the necklace splitting problem, mentioned in the introduction, we may examine similar avoidance problems with respect to anagrams. Given a tangram $T$, one may define its \emph{split number} $\alpha(T)$ as the least number of cuts needed to decompose it into factors that can be made into a pair of anagrams. For instance, $\alpha(abcacb)=1$, while $\alpha(aabbcc)=3$. By the necklace splitting theorem \cite{Alon-West,Goldberg-West}, we know that $\alpha(T)\leqslant q$ holds for any tangram $T$ over an alphabet with $q$ letters, which is best possible.

Now, for a fixed $k\geqslant1$, one may define \emph{$k$-anagram-free} words as those words whose all factors $F$ satisfy $\alpha(F)\geqslant k+1$. In particular, $1$-anagram-free words coincide with words avoiding \emph{abelian squares} (see \cite{FiciPuzynina}). Let $a(k)$ denote the least size of an alphabet needed to construct arbitrarily long $k$-anagram-free words.

Extending theorem of Thue \cite{Thue}, and solving a problem posed by Erd\H{o}s \cite{Erdos}, Ker\"{a}nen \cite{Keranen} constructed an infinite word over four letters without abelian squares. So, we know that $a(1)=4$. This is currently the only known value of this function.

\begin{conjecture}
	For every $k\geqslant 1$, we have $a(k)\leqslant k+3$.
\end{conjecture}

Actually this problem was stated (as a question) by Alon, Grytczuk, Michałek, and Lasoń in \cite{AlonGML}, where they considered a continuous version of anagram-free words. They proved there that an analogous inequality holds for measurable colorings of the real line $\mathbb{R}$.

\end{document}